\documentclass[10pt, reqno]{amsart}

\usepackage[applemac]{inputenc}
\usepackage{amsmath}
\usepackage{amssymb}
\usepackage{mathrsfs}
\usepackage{graphics}
\usepackage[dvipdf]{graphicx}
\usepackage{geometry}
\usepackage{setspace}
\usepackage{amsthm}
\usepackage{paralist}
\usepackage{dsfont}
\usepackage{verbatim}
\usepackage{bm}
\usepackage{hyperref}

\usepackage[numbers]{natbib}

\renewcommand{\tilde}{\widetilde}
\renewcommand{\bar}{\overline}

\theoremstyle{definition}
\newtheorem{definition}{Definition}[section]
\newtheorem{remark}[definition]{Remark}
\newtheorem{assumption}[definition]{Assumption}
\theoremstyle{plain}
\newtheorem{lemma}[definition]{Lemma}
\newtheorem{thm}[definition]{Theorem}
\newtheorem{prop}[definition]{Proposition}

\numberwithin{equation}{section}

\newcommand\R{\mathbb{R}}
\newcommand\Ru{{\mathbb{R} \cup \left\{\infty\right\}}}
\newcommand\Rd{{\mathbb{R}^d}}

\newcommand\ninN{{n\in\mathbb{N}}}
\newcommand\kinN{{k\in\mathbb{N}}}
\newcommand\ntoinf{{n\rightarrow \infty}}

\newcommand\rk{\dens^{k}_{\tau}}
\newcommand\rkm{\dens^{k\minus 1}_{\tau}}
\newcommand\rkzm{\dens^{k\minus 2}_{\tau}}

\newcommand\rN{\dens^{N}_{\tau}}

\newcommand{\nrg}{{\mathcal{F}}}
\newcommand{\tW}{{\mathcal{W}}}
\newcommand{\tV}{\mathcal{V}}
\newcommand{\heat}{{\mathcal{H}}}
\newcommand{\tU}{{\mathcal{U}}}
\newcommand{\calA}{\mathcal{A}}
\newcommand{\W}{\mathbf{W}_2}
\newcommand{\M}{\mathbf{M}_2}

\newcommand{\dens}{\rho}
\newcommand{\denses}{{\mathscr{P}_2(\Omega)}}
\newcommand{\Leins}{{L^1(\Omega)}}
\newcommand{\Lm}{{L^m(\Omega)}}
\newcommand{\BV}{BV(\Omega)}
\newcommand{\Km}{\mathcal{K}_m}

\newcommand{\minus}{\scalebox{0.4}[0.9]{$-$}}

\newcommand{\N}{{\mathbb{N}}}

\newcommand{\bp}{{\bm{p}}}
\newcommand{\bq}{\bm{q}}
\newcommand{\X}{{\bm{X}}}
\newcommand{\dd}{\,\mathrm{d}}

\begin{document}


\begin{abstract}
We prove convergence of a variational formulation of the BDF2 method applied to the non-linear Fokker-Planck equation.
Our approach is inspired by the JKO-method and exploits the differential structure of the underlying $L^2$-Wasserstein space.
The technique presented here extends and strengthens the results of our own recent work on the BDF2 method for general metric gradient flows in the special case of the non-linear Fokker-Planck equation: 
firstly, we do not require uniform semi-convexity of the augmented energy functional;
secondly, we prove strong instead of merely weak convergence of the time-discrete approximations;
thirdly, we directly prove without using the abstract theory of curves of maximal slope that the obtained limit curve is a weak solution of the non-linear Fokker-Planck equation.
\end{abstract}


\title[A BDF2-Approach for the non-linear Fokker-Planck Equation]{A BDF2-Approach for the non-linear Fokker-Planck Equation}
\author[Simon Plazotta]{Simon Plazotta}
\address{Zentrum f\"ur Mathematik \\ Technische Universit\"at M\"unchen \\ 85747 Garching, Germany}
\email{plazotta@ma.tum.de}
\thanks{This research has been supported by the German Research Foundation (DFG), SFB TRR 109. The authors would like to thank Daniel Matthes for helpful discussions and remarks.}
\date{\today}
\maketitle


\section{Introduction}\label{sec:Intro}
This article is concerned with the proof of well-posedness and convergence of a formally higher-order semi-discretization in time, inspired by the Backward Differentiation Formula 2 (BDF2), applied to the \emph{non-linear Fokker-Planck equation with no-flux boundary condition}:
\begin{align}	\label{eq:FP}
\begin{split}
	\partial _t \dens   = \Delta (\dens^m)  + \mathrm{div} \left(\dens \nabla V  \right) + \mathrm{div} \left(\dens \nabla  (W \ast \dens) \right) \qquad \text{in } (0,\infty)\times\Omega , \\
	\bm{n} \cdot \operatorname{D}\dens = 0, \quad \text{on } (0,\infty)\times\partial\Omega, \qquad \qquad	\dens(0,x)  = \dens^0(x) \quad \text{in } \Omega.
\end{split}
\end{align}
We consider \eqref{eq:FP} as an evolutionary equation in the space of probability measures $\denses$ with finite second moment (i.e $\M(\mu) := \int_{\Omega} \left\|x\right\|^2  \dd \mu(x) < \infty$), where $\Omega=\Rd$ or $\Omega \subset\Rd$ is an open and bounded domain with Lipschitz-continuous boundary $\partial\Omega$ and normal derivative $\bm{n}$.
Indeed, if \eqref{eq:FP} is initialized with $\dens^0 \in\denses$ then there exists a weak solution $\dens:[0,\infty)\times\Omega \to \R_{\geq 0}$ such that $\dens(0) =\dens^0$ and $\dens(t) \in \denses$ for each $t>0$.
 
The modern approach towards the theoretical analysis of equation \eqref{eq:FP} is the gradient flow structure in the $L^2$-Wasserstein space $(\denses,\W)$, see \cite{ags,jko,otto2001,santambrogio2015optimal,villani,villani2008optimal}. 
The $L^2$-Wasserstein distance $\W$ between two measures $\mu$ and $\nu$ in $\denses$ is defined by
\begin{align}\label{eq:Wasserstein}
\W^2(\mu,\nu):= \min_{\bp \in \Gamma(\mu,\nu)} \ \  \int_{\Omega^2} \left\|x-y\right\|^2\, \mathrm{d} \bp(x,y),
\end{align}
where $\Gamma(\mu,\nu):= \{ \bp \in \mathscr{P}(\Omega \times \Omega) :(\pi^1)_\# \bp = \mu, \  (\pi^2)_\# \bp = \nu\}$ is the set of all transport plans from $\dens$ to $\nu$.
Note, the minimizers $\bp \in \Gamma(\mu,\nu)$ of $\W(\mu,\nu)$ are called the \emph{optimal transport plans}.
The corresponding energy functional \mbox{$\nrg:\denses\to\Ru$} for \eqref{eq:FP} is given by: 
\begin{align}\label{eq:FPnrg}
	\nrg(\mu):=\begin{cases}
		\int_\Omega \dens \log(\dens) +  V \dens + \frac{1}{2}(W\ast \dens)\dens \dd x & \text{if }m=1, \\
		\int_\Omega \frac1{m-1}\dens^m +  V \dens + \frac{1}{2}(W\ast \dens)\dens \dd x & \text{if }m>1,
	\end{cases}
\end{align}
provided that $\mu= \dens \mathcal{L}^d$ and the integrals on the right-hand side are well-defined otherwise we set $\nrg(\mu)=\infty$.

In the framework of $L^2$-Wasserstein gradient flows, existence of solutions has been shown via the JKO-scheme, named after the authors of \cite{jko}.
This scheme is a variational formulation of the Implicit Euler method given as follows:
for fixed time step size $\tau \in(0,\tau_*)$ construct inductively, starting from $\dens^0$, a sequence of probability measures $(\rk)_\kinN$ as the minimizer of an augmented energy functional:
\begin{align} \label{eq:JKO}
	\rk \in  \underset{\dens\in\denses}{\text{argmin}}  \ \frac1{2\tau} \W^2(\dens,\rkm) + \nrg(\dens).
\end{align}
It is known that the thus obtained discrete gradient approximation converge to a solution of the non-linear Fokker-Planck equation \eqref{eq:FP} as $\tau$ tends to zero. 

Note, this scheme has been similarly applied to a variety of PDEs and systems of PDEs with gradient flow structure in the $L^2$-Wasserstein or in a $L^2$-Wasserstein-like space:
 non-local Fokker-Planck equations \cite{carrillo2011,di2018nonlinear,topaz2006nonlocal};
  Fokker-Planck equations on manifolds \cite{erbar2010heat,sturm2005convex};
  fourth order fluid and quantum models \cite{GOthinfilm,gianazza2009,MatthesMcCannSavare};
	 chemotaxis systems \cite{BCCkellersegel,blanchet2012,zinsl2015exponential};
	   Poisson-Nernst-Planck equations \cite{Kinderlehrer2016};
	    multi-component fluid systems \cite{laurencot2011};
	    Cahn-Hilliard equations \cite{lisini2012};
	    degenerate cross-diffusion systems \cite{matthes2017existence,zinsl2015transport}.

Besides the theoretical use to construct solution for \eqref{eq:FP}, this particular discretization \eqref{eq:JKO} provides also a structure preserving numerical scheme. 
The approximate solution inherits automatically positivity, mass conservation and energy dissipation.
Different approaches to actually compute the minimizers of \eqref{eq:JKO} have been investigated: 
particle schemes \cite{calvez2016blow,carrillo2015numerical,carrillo2016convergence,westdickenberg2010variational}; 
evolving diffeomorphisms \cite{carrillo2015numerical,carrillo2009numerical};
Lagrangian schemes \cite{benamou2016discretization,during2010gradient,gallouet2016lagrangian,junge2015fully,matthes2014convergence,matthes2017convergent}; 
entropic regularization \cite{peyre2015entropic}.
However, it turns out that the application of these schemes to gradient flows in $L^2$-Wasserstein space is intricate, since computing the $L^2$-Wasserstein distance and its gradient is difficult in dimension two or more. \\

We proposed in our own recent work \cite{matthes2017variational} a different variational formulation of a semi-discretization in time, i.e., of the Backward Differentiation Formula 2 (BDF2) method. 
In this context the BDF2 method reads as follows: 
for each sufficiently small time step $\tau \in (0,\tau_*)$, let a pair of initial data $(\dens_\tau^{\minus 1},\dens_\tau^0)$ be given that approximate $\dens_0$. 
Then, define inductively the \emph{discrete solution} $(\rk)_\kinN$ as the minimizers of the following augmented energy functional, 
		\begin{align}\label{eq:BDF2}
				\rk \in 	\underset{\dens\in\denses}{\text{argmin}} \ \frac1{ \tau}  \W^2(\rkm,\dens) - \frac1{4 \tau} \W^2(\rkzm,\dens) +  \mathcal{F}(\dens). 
		\end{align}
Similar to the JKO-scheme the BDF2 method is structure preserving in the sense that the discrete solution inherits automatically positivity, mass preservation and is \emph{almost} energy dissipating (see lemma \ref{lem:EnergyDim}). 
We remark that recently also other variational formulations of formally higher-order time discretizations have been investigated, namely Runge-Kutta methods \cite{laguzet:hal-01404619,LEGENDRE2017345}. \\
 		
Our main contribution in this work is to improve the convergence result of \cite{matthes2017variational} from weak to strong convergence of the discrete solution $(\rk)_\kinN$.
Also in contrast to \cite{matthes2017variational}, our approach is independent of the uniform semi-convexity of the augmented energy functional on the right-hand side of \eqref{eq:BDF2}. 
More in the spirit of the original works on the linear Fokker-Planck equation of Kinderlehrer et al. \cite{jko}, we solely utilize the differential structure of both the $L^2$-Wasserstein space and of the augmented energy functional. 

Note, the BDF2 method and the techniques presented here have two further possible applications. 
Firstly, PDEs with gradient flow structure such that the energy function $\nrg$ do not possesses any uniform semi-convexity property -- like the Hele-Shaw equation seen as $L^2$-Wasserstein gradient flow -- are not covered in \cite{matthes2017variational}. 
However, as long as the subdifferential calculus in the $L^2$-Wasserstein space is applicable to $\nrg$ our method is feasible. 
With this technique at hand on can compute from \eqref{eq:BDF2} the discrete Euler-Lagrange equations for the discrete approximation by variations along solutions of the continuity equation (likwise theorem \ref{DiscInequaDens}).
Hence, passing to the limit as $\tau$ tends to zero could yield directly a distributional solution for the aforementioned class of PDEs without using the abstract theory of curves of steepest descent for \mbox{$\lambda$-contractive} gradient flows.
Secondly, the formally higher-order approximation in time is expected to improve the performance of numerical simulations due to the better resolution of the solution with respect to a coarser time grid. 

\emph{In conclusion, the BDF2 method provides a structure preserving numerical scheme of formally higher-order approximation in time with a strong notion of convergence.} \\

Our main results concerning the well-posedness and the limit behavior as $\tau\searrow 0$ of the \emph{interpolated solution} $\bar\dens_\tau$, which is defined as the piecewise constant interpolation in time of the discrete solution $(\rk)_\kinN$ obtained by the BDF2 method \eqref{eq:BDF2},
\begin{align*}
	\bar\dens_\tau(0)=\dens_\tau^0, \qquad \bar\dens_\tau(t)=\rk \qquad \mathrm{for} \ t\in\left((k-1)\tau, k\tau\right] \qquad \mathrm{and} \ \kinN,
\end{align*}
is stated in the following theorem. The threshold $\tau_*$ is specified in \eqref{eq:taustar}.
\begin{thm}\label{thm:mainthm}
	Let $\Omega \subset \Rd$ be either an open and bounded domain with Lipschitz continuous boundary $\partial \Omega$ or let $\Omega=\Rd$. Further, assume $m\geq1$ and that $V$ and $W$ satisfy Assumption \ref{ass:VW}. Given a vanishing sequence $(\tau_n)_\ninN$ of step sizes $\tau_n\in(0,\tau_*)$ and initial data $(\dens_{\tau_n}^{\minus 1},\dens_{\tau_n}^0)$ satisfying Assumption \ref{ass:I}, then the following hold:
	\begin{enumerate}[(i)]
		\item \textbf{Existence of the discrete solutions.} For each step size $\tau \in(0,\tau_*)$ there exists a sequence $(\rk)_\kinN$ obtained by the BDF2 scheme \eqref{eq:BDF2}, which satisfies the step size independent bounds \eqref{eq:classicalbounds} on the kinetic energy, on the internal energy, and on the second moments.
		\item \textbf{Narrow convergence in $\bm{\denses}$.} There exists a (non-relabelled) subsequence $(\tau_n)_\ninN$ and a limit curve $\dens_*\in\mathrm{AC}^2(0,\infty;(\denses,\W))$ such that for any $t\geq 0$:
		\begin{align*}
			\bar\dens_{\tau_n}(t) \rightharpoonup \dens_*(t) \qquad  &\text{narrowly in the space }\denses \text{ as } \ntoinf.
		\end{align*}
	\item \textbf{Step size independent $\bm{L^2(0,T;BV(\Omega))}$-estimate.} For each fixed time horizon $T>0$ there exists a non-negative constant $C$, depending only on $m,V,W$, and $T$ such that for each $\tau \in(0,\tau_*)$: 
		\begin{align*}
			 \left\| (\bar\dens_\tau)^m\right\|_{L^2(0,T;BV(\Omega))} \leq C.
		\end{align*} 
		\item \textbf{Strong convergence in $\bm{L^m(0,T;L^m(\Omega))}$.} With the notations from (ii), there exists a further (non-relabelled) subsequence $(\tau_n)_\ninN$ such that for all $T>0$:
		\begin{enumerate}
		\item In the case of an open and bounded set $\Omega \subset\Rd$ with Lipschitz-continuous boundary $\partial \Omega$, we have	:	\begin{align*}
			\bar\dens_{\tau_n} \rightarrow \dens_* \qquad &\text{in } L^m(0,T;L^m(\Omega)) \text{ as } n \to\infty.
		\end{align*}
		\item In the case of the entire space, i.e., $\Omega=\Rd$, we have for every open and bounded set $\tilde\Omega \Subset \Rd$:
		\begin{align*}
			\bar\dens_{\tau_n} \rightarrow \dens_* \qquad &\text{in } L^m(0,T;L^m(\tilde\Omega)) \text{ as } n \to\infty.
		\end{align*} 
		\end{enumerate}
	\item \textbf{Solution of the non-linear Fokker-Planck equation.} The limit curve $\dens_*$ from (ii) satisfies the non-linear Fokker-Planck equation with no-flux boundary condition \eqref{eq:FP} in the distributional sense, i.e., we have for each test function $\psi\in C^\infty_c([0,\infty)\times\Omega)$:
	\begin{align*}
		\int_0^\infty \int_\Omega - \Delta\psi \, \dens_*^m + \langle \nabla\psi,\nabla V \rangle \, \dens_*+ \langle \nabla\psi,\nabla W \ast \dens_* \rangle \, \dens_* \dd x \dd t 
		= \int_0^\infty \int_\Omega \partial_t \psi \, \dens_* \dd x\dd t+ \int_\Omega \psi(0) \, \dens^0 \dd x.
	\end{align*}
	\end{enumerate}
\end{thm}

The plan of the paper is as follows.
In section \ref{sec:Pre} we recall the basic notation of the theoretical framework of the gradient flow formulation of the non-linear Fokker-Planck equation, of our particular time-discretization and of $BV(\Omega)$-spaces.
 Section \ref{sec:disc} is concerned with basic properties of the augmented energy functional and of the approximation obtained by that scheme. 
 In Section \ref{sec:DiscEL} we derive the discrete Euler-Lagrange equations by means of a variation of the augmented energy functional along solutions to the continuity equation.
From these discrete Euler-Lagrange equations we derive $BV(\Omega)$-regularity estimates.
 In Section \ref{sec:Convergence} we complete the proof of the main theorem and prove the convergence of the approximation to the distributional solution of the non-linear Fokker-Planck equation \eqref{eq:FP}.


\section{Setup and Assumptions}\label{sec:Pre} 
\subsection{Gradient Flow Framework of the Non-linear Fokker-Planck equation}
\label{subsec:Wasserstein}
Throughout the rest of the paper $\Omega \subset \Rd$ is either equal to $\R^d$ or some open and bounded domain with Lipschitz-continuous boundary $\partial\Omega$.
By $\mathscr{P}(\Omega)$ we will denote the set of probability measures on $\Omega$. 
We say a sequence of measures $(\mu_n)_\ninN$ \emph{converges narrowly} to $\mu \in\mathscr{P}(\Omega)$ if and only if
\begin{align*}
	\lim_{n\to\infty} \int_\Omega \psi \dd \mu_n(x) = \int_\Omega \psi \dd \mu(x) \qquad \text{for all } \psi \in C_b(\Omega),
\end{align*}
i.e., narrow convergence is equal to $weak^*$-$convergence$, which is induced by the pairing of the continuous and bounded functions $C_b(\Omega)$  with the corresponding dual space of finite Borel measures $\mathcal{M}_f(\Omega)$.

A curve $\mu:[0,\infty)\to\denses$ is said to be \emph{$L^2$-absolutely continuous}, we write $\mu\in\mathrm{AC}^2(0,\infty;(\denses,\W))$, if there exists a function $A\in L^2_{loc}(0,\infty)$ such that 
\begin{align*}
	\W(\mu(s),\mu(t))\leq \int_s	^t A(r) \dd r \qquad \text{for all } 0\leq s\leq t. 
\end{align*}

The corresponding energy functional $\nrg$ of the non-linear Fokker-Planck equation \eqref{eq:FP}, defined in \eqref{eq:FPnrg}, is the sum of three parts: \emph{internal energy} $\tU_m$; \emph{external energy} $\tV$; \emph{interaction energy} $\tW$. 
The internal energy $\tU_m$ is given by
\begin{align*}
 \mathrm{for} \ m=1: \quad \heat(\mu) & := \tU_1 (\mu):= \int_\Omega \dens(x) \log (\dens(x)) \dd x, &	\mathrm{or \ for} \ m>1: \quad \tU_m (\mu) & :=  \frac{1}{m-1}\int_\Omega \dens(x)^m  \dd x,  
\end{align*}
where the measure $\mu$ is absolutely continuous with respect to the Lebesgue measure $\mathcal{L}^d$ with density $\rho$, i.e., $\mu= \dens \mathcal{L}^d$. 
For measures $\dens$ which are singular with respect to the Lebesgue measure, we set $\tU_m(\dens) = \infty$. 
This convention makes the internal energy $\tU_m$ lower semi-continuous with respect to narrow convergence, see \cite{ambrosio2000functions}. 
Therefore, by a slight abuse of notation, we shall always identify an absolutely continuous measure $\mu$ with its corresponding density $\dens$. The according \emph{proper domains} of the $\heat$ and $\tU_m$ are given by
	\begin{align*}	
	 \mathcal{K}_1 := &  \left\{ \dens \in \denses \mid \dens \log (\dens)  \in \Leins  \right\}, & 
		\Km  := & \left\{ \dens \in \denses \mid  \dens^m \in \Leins \right\} .
	\end{align*}
	Further, the external energy $\tV$ and the interaction energy $\tW$ are defined via
\begin{align*}
	\tV(\dens)&:=\int_\Omega V \dens \dd x ,  &\tW(\dens)&:= \frac12 \int_\Omega (W\ast\dens) \dens \dd x := \frac12 \int_{\Omega^2}W(x-y)  \dens(y) \dens(x) \dd x \dd y. 
\end{align*}
For the rest of the paper our assumption on the external potential $V$ and on the interaction kernel $W$ reads as follows:\\[-11mm]
\begin{quote}%
	\begin{assumption}\label{ass:VW} Let the external potential $V\in\mathcal{C}^{1} \left( \Omega \right)$ and the symmetric interaction kernel $W\in\mathcal{C}^{1} \left( \Rd \right)$ be bounded as follows:
		\begin{align*}
			\left| V(x) \right|, \, \left| W(x) \right|, \, \left\| \nabla V(x) \right\| , \, \left\| \nabla W(x) \right\|  & \leq  d_1 \left( 1 + \left\|x\right\|^2 \right).
			\end{align*}
	\end{assumption}	
\end{quote}
Note, this standard assumption guarantees that 	all integrals with respect to any measure $\dens\in\denses$ and with integrands $V$, $W$, $\nabla V$, or $\nabla W$ are well-defined and finite.
Further, the functionals $\tV$ and $\tW$ are continuous with respect to narrow convergence by this assumption \cite{ags}.

 \subsection{Discretization}\label{subsec:disc}
Similarly to \cite{matthes2017variational}, the \emph{BDF2 penalization} $\Psi:(0,\infty)\times(\denses)^3 \rightarrow \Ru$ of the original energy functional $\nrg$ is defined by
	\begin{align*}
					\Psi(\tau,\eta,\nu; \cdot):  \denses \rightarrow \Ru; \ \Psi(\tau,\eta, \nu ; \dens):= \frac1{ \tau}  \W^2(\nu,\dens) - \frac1{4 \tau} \W^2(\eta,\dens) +  \mathcal{F}(\dens). 
	\end{align*}
With this notation, given a time step size $\tau \in(0,\tau_*)$ and a pair of initial data $(\dens_\tau^{\minus 1}, \dens_\tau^0)$, the \emph{discrete solution} $(\rk)_{\kinN}$ for $\nrg$ on $(\denses,\W)$ defined in \eqref{eq:BDF2} is equivalently defined by the recursive formula 
\begin{align}
	\rk  \in \underset{\dens\in\denses}{\mathrm{argmin}} \ \Psi(\tau,\rkzm,\rkm;\dens) \quad \text{for} \ \kinN.	\label{eq:BDF2scheme}
\end{align}		
In the rest of the paper we approximate $\dens^0$ for a given time step size $\tau \in(0,\tau_*)$ by a pair of initial data $(\dens_\tau^{\minus 1}, \dens_\tau^0)$ as follows:\\[-11mm]%
	 \begin{quote}
 	\begin{assumption}\label{ass:I} 
	There are non-negative constants $d_3,d_4$ such that for all $\tau\in(0,\tau_*)$:
			\begin{enumerate}[({I}1)]
		\item $\W^2(\dens_\tau^{\minus 1},\dens_\tau^0)\leq d_3 \tau$ and $\W^2(\dens_\tau^0,\dens^0)\leq d_3 \tau$.
		\item $ \tU_m(\dens_\tau^{\minus 1}) \leq d_4$ and $\tU_m(\dens_\tau^0) \leq d_4.$
	\end{enumerate}
\end{assumption} 
 \end{quote}

\subsection{Functions of Bounded Variation}\label{subsec:bv}
We recall the basic definitions and properties of functions of bounded variation, following \cite{giusti1984minimal}.
A function $\dens\in L^1(\Omega)$ is called a \emph{function of bounded variation} if and only if
\begin{align*}
		V(\dens,\Omega):= \sup \left\{ \int_\Omega \dens(x) \,\operatorname{div} \xi(x)  \dd x \mid \xi \in C^\infty_c(\Omega,\Rd), \ \left\| \xi \right\|_\infty \leq 1  \right\} < \infty.
	\end{align*}
The set of all functions of bounded variation is denoted by $BV(\Omega)$ and can be equipped with the norm:
\begin{align*}
		\left\|\dens\right\|_{\BV} = \left\|\dens\right\|_\Leins + V(\dens,\Omega).
\end{align*}
For open sets $\Omega\subset\Rd$ the set $BV(\Omega)$ is a Banach space and the norm is lower semi-continuous with respect to the weak convergence in $L^1(\Omega)$.
In case that $\Omega$ is an open and bounded set in $\Rd$ with Lipschitz-continuous boundary $\partial\Omega$, sets of functions uniformly bounded in the $BV(\Omega)$-norm are relatively compact in $L^1(\Omega)$, see \cite[Theorem 1.19]{giusti1984minimal} for the statement and the proof.

\section{Well-posedness and Basic Properties of the BDF2 Scheme}\label{sec:disc}

\subsection{Lower Bounds and Lower Semi-Continuity}\label{subsec:Basic}		

We establish the following two basic properties of the BDF2 penalization $\Psi$, which will be essential for the solvability of problem \eqref{eq:BDF2scheme}: $\Psi(\tau,\eta,\nu; \, \cdot \,)$ is bounded from below and lower semi-continuous. 

\begin{lemma}[Lower Bound]\label{lem:lower}
		There exist a non-negative constant $d_2$ such that the BDF2 penalization $\Psi$ satisfies for each $\tau>0$ and for all $\dens,\eta,\nu \in \denses$:
	\begin{align}\label{eq:lowerbound}
		\Psi(\tau,\eta,\nu; \dens) 	\geq &   \left( \frac1{8\tau} - \frac32 d_1 -d_2 \right) \M(\dens)  - \frac1{ \tau} \M (\nu)  - \frac3{4 \tau} \M(\eta) - d_2 - \frac32 d_1.
	\end{align}
\end{lemma}

\begin{remark} 
Without loss of generality we assume that 
\begin{align}\label{eq:taustar}
	\tau_* < (12d_1+8d_2)^{-1},
\end{align}
such that $\dens \mapsto \Psi(\tau,\eta,\nu; \dens)$ is bounded from below by a constant.
\end{remark}

\begin{proof} Without loss of generality we can assume $\dens$ is an absolutely continuous measure with density $\dens$. Observe that $\heat$ is not bounded from below by a constant on $\denses$. However, we derive from the Carleman estimate a lower bound of $\heat$ in terms of the second moment $\M$, see \cite{jko}, i.e., there exist non-negative constants $d_2\geq0$ and $\gamma \in (\frac{d}{d+2} ,1 )$ such that 
	\begin{align}\label{eq:carleman}
		 \mathcal{U}_m (\dens)  \geq \mathcal{H}(\dens) \geq  - d_2 ( 1+ \M(\dens)  )^{\gamma}\geq -d_2(1+\M(\dens)). 
	\end{align}
Since the external potential $V$ and the interaction kernel $W$ grow at most quadratically at infinity, the corresponding energies can be estimated from below in terms of the second moment $\M$ by
	\begin{align*}
		\tV(\dens) + \tW(\dens) \geq - d_1\int_\Omega (1+\left\|x\right\|^2 )  \dd \dens(x) - \frac12 d_1 \int_{\Omega^2} (1+ \left\|x-y\right\|^2)\dd \dens(x)  \dd \dens(y) =  - \frac32 d_1(1 + \M(\dens)).
	\end{align*}
	From the elementary inequality $\left\|x\right\|^2-2\left\|y\right\|^2\leq2\left\|x-y\right\|^2\leq 3\left\|x\right\|^2 + 6\left\|y\right\|^2$ and from the definition of $\W$ it follows immediately
	\begin{align}\label{eq:bounddist}
		 \M( \dens) - 2\M (\nu) \leq 2\W^2 (\dens,\nu) \leq 3\M(\dens) + 6 \M(\nu) \qquad \text{for all } \dens,\nu\in\denses.
	\end{align}	
	Combining all three inequalities, we can deduce the following lower bound: 
	\begin{align*}
		\Psi(\tau,\eta,\nu; \dens)\geq & \frac1{2\tau} \M(\dens) - \frac1{ \tau} \M (\nu) - \frac3{8\tau} \M(\dens) - \frac3{4 \tau} \M(\eta) - d_2 ( 1+ \M(\dens)  ) - \frac32 d_1(1 + \M(\dens)),			\end{align*}
	which is equivalent to the desired inequality \eqref{eq:lowerbound}.	
\end{proof}

\begin{lemma}[Lower semi-continuity]\label{lem:lsc}
	For each $\tau>0$ and for all $\eta,\nu \in \denses$ the BDF2 penalization $\Psi(\tau,\eta,\nu;\, \cdot \,)$ is lower semi-continuous with respect to narrow convergence.
\end{lemma}		
			
\begin{proof}			
Due to the lower semi-continuity with respect to narrow convergence of the internal energy $\tU_m$, the external potential $\tV$, and the interaction energy $\tW$, the energy $\nrg$ is also lower semi-continuous with respect to narrow convergence as sum of lower semi-continuous functions. 

	 Thus it remains to prove the lower semi-continuity of the auxiliary functional $\calA: \denses \to \R$, defined via
	 \begin{align*}
	 \calA(\dens):=4 \W^2(\nu,\dens)-\W^2(\eta,\dens).
	 \end{align*}
	 First, we simplify the auxiliary functional $\calA$. 
	 Let $\bp^1 \in\Gamma(\dens,\nu)$ and $\bp^2 \in\Gamma(\dens,\eta)$ be two optimal transport plans.
	 Further, introduce the special three-plan $\bp \in \Gamma(\dens,\nu,\eta):=\{ \bp \in \mathscr{P}(\Omega \times \Omega\times \Omega) :(\pi^1)_\# \bp = \dens, \  (\pi^2)_\# \bp = \nu, \  (\pi^3)_\# \bp = \eta\}$ such that $\bp$ has marginal with respect to the $x$- and $y$-components equals to $\bp^1$ and the marginal with respect to the $x$- and $z$-components is equal to $\bp^2$, i.e., $(\pi^{(1,2)})_\# \bp = \bp^1$ and $(\pi^{(1,3)})_\# \bp = \bp^2$.
	 The existence of such a three-plan is guaranteed by the gluing lemma, see \cite[Lemma 5.3.2]{ags}.
	Then we can rewrite the auxiliary functional $\calA$ as
	\begin{align}\label{eq:lsc2}
		\calA(\dens) =\int_{\Omega^2} 4\left\|x-y\right\|^2 \dd \bp^1(x,y) - \int_{\Omega^2} \left\|x-z\right\|^2 \dd \bp^2(x,z) = \int_{\Omega^3} 4\left\|x-y\right\|^2  -  \left\|x-z\right\|^2 \dd \bp(x,y,z).
	\end{align} 
	 
	 Now, let $(\dens_n)_\ninN$ be a narrowly converging sequence with limit $\dens_*\in\denses$. 
	 Since $(\dens_n)_\ninN$ is narrowly converging to $\dens_*$, the sequences $(\bp_n^1)_\ninN$ and $(\bp_n^2)_\ninN$ are relatively compact in $\mathscr{P}_2(\Omega^2)$ with respect to narrow convergence and any limit point is an optimal transport plan, see \cite[Proposition 7.1.3]{ags}. 
	 Thus we can extract a non-relabelled subsequence such that $(\bp_n^1)_\ninN$ and $(\bp_n^2)_\ninN$ converge narrowly to an optimal transport plan $\bp_*^1\in\Gamma(\dens_*,\nu)$ and to an optimal transport plan $\bp_*^2\in\Gamma(\dens_*,\eta)$, respectively. 
	 By the same argument, the sequence $(\bp_n)_\ninN$ of three-plans is relatively compact in $\mathscr{P}_2(\Omega^3)$ with respect to narrow convergence. 
	 Therefore we can extract a further non-relabelled subsequence such that $(\bp_n)_\ninN$ narrowly converges to some three-plan $\bp_*\in \Gamma(\dens_*,\nu,\eta)$.
	Taking marginals is continuous with respect to narrow convergence, so we have $(\pi^{(1,2)})_\#\bp_*=\bp_*^1$ and $(\pi^{(1,3)})_\#\bp_*=\bp_*^2$, i.e., this limit three-plan $\bp_*$ is admissible in \eqref{eq:lsc2}.
	
	Next, we want to apply the lower semi-continuity result \cite[Lemma 5.1.7]{ags} to the alternative representation of $\calA$.
	The uniform integrability of the negative part of the integrand in \eqref{eq:lsc2} with respect to $(\bp_n)_\ninN$ in the sense of \cite{ags} follows by the elementary inequality 
	 \begin{align*}
	 	4\left\|x-y\right\|^2 - \left\|x-z\right\|^2 \geq \frac12 \left\|x\right\|^2 -4\left\|y\right\|^2-3\left\|z\right\|^2 \geq -4\left(\left\|y\right\|^2 + \left\|z\right\|^2\right).
	 \end{align*}
	  Thus the lower bound on $4\left\|x-y\right\|^2 - \left\|x-z\right\|^2$ is independent of $x$.
	  Since the second moments of $\nu$ and $\eta$ are finite that difference is uniform integrable with respect to the family $(\bp_n)_\ninN$.	 
	 Hence, we can invoke \cite[Lemma 5.1.7]{ags} to conclude 
	\begin{align*}
	 \int_{\Omega^3} 4 \left\|x-y\right\|^2 - \left\|x-z\right\|^2 \dd \bp_*(x,y,z) \leq \liminf_\ntoinf \int_{\Omega^3} 4\left\|x-y\right\|^2 - \left\|x-z\right\|^2 \dd \bp_n(x,y,z).
	 \end{align*}
 Therefore the auxiliary function $\dens \mapsto \calA(\dens)=4 \W^2(\nu,\dens)-\W^2(\eta,\dens)$ is lower semi-continuous with respect to narrow convergence.
	 \end{proof}
			
\subsection{Existence of Minimizer}\label{subsec:Existence}

Recall that the well-posedness of a single step of the BDF2 scheme is equivalent to the existence of a minimizer in \eqref{eq:BDF2scheme}. 
The augmented energy functional $\Psi$ shares no uniform semi-convexity as in the case of \cite{matthes2017variational}, so we cannot exploit the convexity to ensure the existence of a minimizer. 
Nevertheless, a standard technique from the calculus of variations yields the existence of a minimizer. 

\begin{thm}[Existence of a minimizer] \label{ExistenceBDF2}
For each $\tau \in \left(0 , \tau_* \right)$ and for all $ \eta,\nu \in \denses$, there exists an absolutely continuous minimizer $\dens_* \in \Km$ of the map $\dens \mapsto \Psi(\tau,\eta,\nu; \dens)$.
\end{thm}

\begin{proof}
	Take a minimizing sequence $(\dens_n)_\ninN$ for the BDF2 penalization $\dens \mapsto \Psi(\tau,\eta,\nu; \dens)$.
	To extract a convergent subsequence, we use the auxiliary inequality \eqref{eq:lowerbound}. 
	Since $ \tau < \tau_*$, the pre-factor of the second moment $\M(\dens)$ in \eqref{eq:lowerbound} is positive. 
	Hence, the second moment $(\M(\dens_n))_\ninN$ of the minimizing sequence $(\dens_n)_\ninN$ is bounded. 
	Also the internal energy $\tU_m(\dens_n)$ of the minimizing sequence is bounded, since
	\begin{align*}
		\tU_m(\dens_n) \leq &  \Psi(\tau,\eta,\nu;\dens_n) + \frac{1}{4\tau}\W^2(\eta,\dens_n) -\tV(\dens_n) - \tW(\dens_n) 
		 \leq  \sup_\ninN \left[\Psi(\tau,\eta,\nu;\dens_n) + C(1+\M(\dens_n)) \right] < \infty.
	\end{align*}
	Due to the super-linear growth of $\dens\mapsto\dens\log(\dens)$ and of $\dens\mapsto\dens^m$, we can apply the Dunford-Pettis Theorem to the densities $(\dens_n)_\ninN$ and we can extract a non-relabelled subsequence $(\dens_n)_\ninN$ converging weakly in $L^1(\Omega)$. 
	Since $C^b(\Omega)\subset L^\infty(\Omega)=(L^1(\Omega))^*$, in this case  we can deduce from the weak convergence in $L^1(\Omega)$ of the sequence of densities the narrow convergence of the corresponding measures.
	 Summarized, the sequence $(\dens_n)_\ninN$ also converges narrowly to an absolutely continuous measure $\dens_*\in\denses$ with density $\dens_*$.
	 By the lower semi-continuity of the $L^m(\Omega)$-norm with respect to narrow convergence it follows $\dens_*\in\Km$.
	 
	 To prove that $\dens_*$ is indeed a minimizer we use the lower semi-continuity of the BDF2 penalization $\Psi$, proven in Lemma \ref{lem:lsc}, to conclude
	 \begin{align*}
	 	\Psi(\tau,\eta,\nu; \dens_*)\leq \liminf_{\ntoinf} \Psi(\tau,\eta,\nu; \dens_n) = \inf_{\dens\in\denses} \Psi(\tau,\eta,\nu; \dens).
	 \end{align*}
		Indeed, the limit measure with density $\dens_*$ is a minimizer of the BDF2 penalization $\Psi(\tau,\eta,\nu;\,\cdot\,)$.
	\end{proof}

\subsection{Step size independent estimates}\label{subsec:ClassEst}
By the previous theorem, the sequence $(\rk)_{\kinN}$ given by the BDF2 method is well-defined for $\tau\in(0,\tau_*)$. 
Next, we deduce three step size independent bounds: \emph{on the kinetic energy, on the internal energy, and  on the second moment}. 
We want to emphasize that these estimates are intrinsic properties of the scheme, which do not rely on any uniform semi-convexity of the augmented energy functional $\Psi$. 
The original proof of those estimates can be found in \cite{matthes2017variational} and for the sake of the completeness we recall a proof in Appendix \ref{app:Bounds} adapted to the $L^2$-Wasserstein formalism.

\begin{thm}[Classical estimates] \label{thm:bounds}
Fix a time horizon $T>0$.
There exists a constant $C$, depending only on $d_1$ to $d_4$ and $T$, 
  such that the corresponding discrete solutions $(\dens_\tau^k)_{k\in\N}$ satisfy
  \begin{align}
    \sum_{k=0}^{N} \frac{1}{2\tau} \W^2(\rkm,\rk) & \leq  C ,  &   |\tU_m(\dens_\tau^{N}) | & \leq C,  & \M(\rN) &\leq C \label{eq:classicalbounds},
  \end{align}
  for all $\tau\in(0,\tau_*)$ and for all $N\in\N$ with $N\tau\le T$. 
 \end{thm}
 
 \begin{proof}
 	The proof of this theorem is given in Appendix \ref{app:Bounds}. 
 \end{proof}

\subsection{Narrow Convergence}

We are able to prove our first weak convergence results.
The step size independent bounds \eqref{eq:classicalbounds} and the Arzel\`a-Ascoli theorem, which can be found in \cite[Proposition 3.3.1]{ags}, guarantee the narrow convergence of the interpolated solution $\bar\dens_\tau$.
 
 \begin{thm}[Narrow convergence in $\denses$]\label{thm:convn} 
	Given a vanishing sequence $(\tau_n)_\ninN$ of step sizes $\tau_n\in(0,\tau_*)$.
	Then, there exists a (non-relabelled) subsequence $(\tau_n)_\ninN$ and a $L^2$-absolutely continuous limit curve $\dens_* \in \mathrm{AC}^2(0,\infty; (\denses,\W))$ such that for any $t\geq0$:
		\begin{align*}
			\bar\dens_{\tau_n}(t) \rightharpoonup \dens_*(t) \qquad  &\text{narrowly in the space }\denses \text{ as } \ntoinf.
		\end{align*}
\end{thm}

\begin{proof}
Fix $T>0$ and define the auxiliary function $A_n\in L^2(0,T)$, also called discrete derivative, as
\begin{align*}
	A_n(t):= \frac{\W(\dens^{k\minus 1}_{\tau_n},\dens^k_{\tau_n})}{{\tau_n}} \quad \text{for} \quad t\in((k-1){\tau_n},k{\tau_n}] \quad \text{and} \quad \kinN.
\end{align*}
Using the step size independent bounds \eqref{eq:classicalbounds} we obtain for $N_T=\max\{N\mid N{\tau_n}\leq T\}$:
\begin{align*}
	\int_0^{T} A_n^2(t) \dd t \leq \sum_{k=1}^{N_T} \int_{(k-1){\tau_n}}^{k{\tau_n}} \left( \frac{\W(\dens^{k\minus1}_{\tau_n},\dens^k_{\tau_n})}{{\tau_n}} \right)^2 \dd t= \sum_{k=1}^{N_T} \frac{\W^2(\dens^{k\minus1}_{\tau_n},\dens^k_{\tau_n})}{{\tau_n}} \leq C.
\end{align*}
Indeed, $A_n\in L^2(0,T)$ and the $L^2(0,T)$-norm of $A_n$ is uniformly bounded independently of the step size ${\tau_n}$. 
Therefore, the sequence $A_n$ possesses a non-relabelled subsequence weakly convergent in $L^2(0,T)$ with limit $A\in L^2(0,T)$. 
To derive an uniform H\"older-estimate for $\bar\dens_{\tau_n}$, choose $0\leq s\leq t \leq T$ arbitrary and define $k_t=\max\{k \in \N \mid k{\tau_n} \leq t\}$, then 
\begin{align}
	\W(\bar\dens_{\tau_n}(s),\bar\dens_{\tau_n}(t))\leq \sum_{k=k_s+1}^{k_t} \W(\dens^{k\minus1}_{\tau_n},\dens^k_{\tau_n}) = \sum_{k=k_s+1}^{k_t}  \int_{(k-1){\tau_n}}^{k{\tau_n}} \frac{\W(\dens^{k\minus 1}_{\tau_n},\dens^k_{\tau_n})}{{\tau_n}}  \dd t \leq \int_{(s- {\tau_n})^+}^{t} A_n(t) \dd t. \label{eq:hcont}
\end{align}
Taking the limit $\ntoinf$ yields, together with the weak convergence in $L^2(0,T)$ of $A_n$ to $A$,
\begin{align*}
	\limsup_{\ntoinf} \W(\bar\dens_{\tau_n}(s),\bar\dens_{\tau_n}(t)) \leq \int_s^t A(r) \dd r.
\end{align*}
Moreover, the second moments of the discrete solutions $(\dens_{\tau_n}^k)_\kinN$ are uniformly bounded independently of the step size $\tau_n$ and therefore the interpolated solutions $\bar\dens_{\tau_n}$ is uniformly contained in a set $K$ which is compact with respect to narrow convergence.
Hence, we can apply the Arzel\`a-Ascoli Theorem \cite[Proposition 3.3.1]{ags} yielding the existence of a non-relabelled subsequence and a limit curve $\dens_*:[0,T]\to\denses$ such that $\bar\dens_{\tau_n}(t)$ converges narrowly to $\dens_*(t)$ for each fixed $t\in[0,T]$. 
Additionally, the limit curve $\dens_*$ is $L^2$-absolutely continuous with modulus of continuity $A \in L^2(0,T)$. A further diagonal argument in $T\to\infty$ yields the narrow convergence on for any $t\geq 0$ and $\dens_*\in\mathrm{AC}^2(0,\infty;(\denses,\W))$.
\end{proof}

\section{Discrete Euler Lagrange Equation and Improved Regularity}\label{sec:DiscEL}
In theorem \ref{DiscInequaDens}, we derive the discrete Euler-Lagrange equations for the weak formulation of the non-linear Fokker-Planck equation \eqref{eq:FP}.
The key idea is the JKO-method introduced in \cite{jko}, i.e., we determine the first variation of the augmented energy functional $\Psi(\tau,\rkzm,\rkm;\, \cdot\, )$ in the space $(\denses,\W)$ along solutions to the continuity equation
\begin{align}\label{eq:cont}
	\partial_s \dens_s + \operatorname{div}(  \xi \, \dens_s ) = 0, \qquad \dens_0=\rk,
\end{align}
for an arbitrary smooth vector field $\xi \in C^\infty_c(\Omega,\Rd)$.
The solution $\dens_s$ is explicitly given by the push-forward of $\rk$ under the flow $\Phi_s$, i.e., $\rho_s = (\Phi_s)_\# \rk$, such that the flow $\Phi_s$ satisfies the initial value problem:
\begin{equation*}
 \frac{\mathrm{d}}{\mathrm{d}s}  \Phi_s (x) = \xi \left( \Phi_s (x) \right), \qquad \Phi_0(x) = x.
\end{equation*}
Note that the flow $\Phi_s$ exists and each $\Phi_s$ is a diffeomorphism on $\Omega$. Additionally, we can calculate the derivative of $\det(\operatorname{D} \Phi_s)$ and we have an explicit representations of the perturbed density $\dens_s$, i.e.,
\begin{align}\label{eq:repcont}
	\frac{\dd }{\dd s} \left[\det(\operatorname{D}\Phi_s(x))\right]_{s=0} &= \operatorname{tr}( \operatorname{D} \xi \circ \Phi_0)= \operatorname{div} (\xi), \qquad \text{and} \qquad  \det(\operatorname{D}\Phi_s) \dens_s \circ \Phi_s = \rk.
\end{align}

\begin{thm}[Discrete Euler-Lagrange equations] \label{DiscInequaDens}
The discrete solution $\left(\rk\right)_\kinN$ obtained by the BDF2 method satisfies for each $k\in\N$ and for all vector fields $\xi \in C^\infty_c(\Omega,\Rd) $
\begin{align}
\begin{split}
 0= & \int_{\Omega}  - \operatorname{div} (\xi) \, (\rk)^m +  \langle \xi , \nabla V\rangle \, \rk + \langle \xi , \nabla W \ast \rk \rangle \, \rk \dd x \\
   & + \frac{2}{\tau} \int_{\Omega^2} \langle \xi(x), x-y \rangle  \dd \bp_\tau^k(x,y) -\frac{1}{2\tau} \int_{\Omega^2} \langle \xi(x),x-z \rangle \dd \bq_\tau^k(x,z)   ,  \label{eq:DiscreteEL}
\end{split}	
\end{align}
where $\bp_\tau^k\in\Gamma(\rk,\rkm)$ and $\bm{q}_\tau^k\in\Gamma(\rk,\rkzm)$ are optimal transport plans.
\end{thm}

\begin{proof}
Fix $\rk, \rkm, \rkzm$ and $\xi\in C^\infty_c(\Omega,\Rd)$. We consider the perturbation $\dens_s$ of $\rk$ as the solution of the continuity equation with velocity field $\xi$ starting at $\rk$, i.e., $\dens_s$ is the solution of \eqref{eq:cont}.
To actually compute the first variation of the heat energy $\heat$ we use exact value of the derivative of $\det(\operatorname{D}\Phi_s)$ and the explicit representation of the perturbed density $\dens_s$, given in \eqref{eq:repcont}, to obtain as limit of the difference quotient
\begin{align*}
\frac{\dd}{\dd s}\left[\heat(\dens_s)\right]_{s=0} =  \lim_{s\to 0}	\frac1{s}\left( \heat(\dens_s) - \heat(\rk) \right) = - \lim_{s\to 0} \int_\Omega  \frac1{s}\log(\det(\operatorname{D}\Phi_s(x))) \, \rk(x) \dd x = - \int_\Omega \operatorname{div} (\xi) \, \rk \dd x.
\end{align*}
Similarly, we can compute the first variations of the internal energy $\tU_m$ for $m>1$, the external potential $\tV$, and the interaction energy $\tW$. The first variation of the energy $\nrg$ along the solution to the continuity equation amounts to
\begin{align}
	\frac{\dd}{\dd s}\left[\nrg(\dens_s) \right]_{s=0}  = \int_\Omega - \operatorname{div} (\xi) \, (\rk)^m + \langle \xi , \nabla V \rangle \,  \rk + \langle \xi , \nabla W \ast \rk \rangle \,  \rk \dd x. \label{eq:discaux1}
\end{align}
The differentiability of the quadratic $L^2$-Wasserstein distance $\W$ along the solution $\dens_s$ of the continuity equation is more technical, for the proof we refer to \cite{ags,villani}. 
Since $\rkzm,\rkm,\rk,\dens_s$ are all absolutely continuous measures, Theorem 8.13 from \cite{villani} is applicable and we can conclude:
\begin{align}
	\frac{\dd}{\dd s}\left[4\W^2(\rkm,\dens_s)-\W^2(\rkzm,\dens_s)\right]_{s=0}= 8 \int_{\Omega^2} \langle \xi(x),x-y\rangle  \dd \bp_\tau^k(x,y) - 2\int_{\Omega^2} \langle \xi(x),x-z\rangle  \dd \bq_\tau^k(x,z), \label{eq:discaux2}
\end{align}
where $\bp_\tau^k\in\Gamma(\rk,\rkm)$ and $\bm{q}_\tau^k\in\Gamma(\rk,\rkzm)$ are optimal transport plans.
Since $\rk$ is a minimizer of the BDF2 penalization $\Psi(\tau,\rkzm,\rkm;\, \cdot\, )$ and since $s\mapsto \Psi(\tau,\rkzm,\rkm;\dens_s )$ is differentiable at $s=0$,
\begin{align*}
	0  =  \frac{\dd}{\dd s} \left[\Psi(\tau,\rkzm,\rkm;\dens_s)\right]_{s=0} 
	= & \frac{1}{4\tau} \frac{\dd}{\dd s}\left[4\W^2(\rkm,\dens_s)-\W^2(\rkzm,\dens_s)\right]_{s=0} + \frac{\dd}{\dd s}\left[\nrg(\dens_s) \right]_{s=0} \\
	 = & \frac{2}{\tau} \int_{\Omega^2} \langle \xi(x), x-y \rangle  \dd \bp_\tau^k(x,y) -\frac{1}{2\tau} \int_{\Omega^2} \langle \xi(x),x-z \rangle \dd \bq_\tau^k(x,z) \\
   & +  \int_{\Omega}  - \operatorname{div} (\xi) \, (\rk)^m +  \langle \xi , \nabla V\rangle \, \rk + \langle \xi , \nabla W \ast \rk \rangle \, \rk \dd x.
\end{align*}
Indeed, we have the desired equality \eqref{eq:DiscreteEL}. 
\end{proof}

The already obtained regularity results for the interpolated solution $\bar\dens_\tau$ are not sufficient to pass to the limit in the first term of the discrete Euler-Lagrange equation \eqref{eq:DiscreteEL}.
Nevertheless, the following bounds in the $BV(\Omega)$-norm of $(\rk)^m$ are sufficient to obtain the desired regularity results.
These estimates can be derived from the discrete Euler-Lagrange equation quite naturally. 

\begin{prop}[Step size independent $BV(\Omega)$-estimate]\label{prop:bv}
Fix a time horizon $T>0$.
  There exists a constant $C$, depending only on $d_1$ to $d_4$ and $T$, 
  such that the corresponding discrete solutions $(\dens_\tau^k)_{k\in\N}$ satisfy for all $\tau\in(0,\tau_*)$ and for all $\kinN$ with $k \tau\leq T$:
	\begin{align}\label{eq:bv}
		\left\|(\rk)^{m}\right\|_{\BV} &\leq C \left(1 + \frac{\W(\rk,\rkm)}{\tau} + \frac{\W(\rk,\rkzm)}{\tau} \right).
	\end{align}
\end{prop}

\begin{proof} 
The $L^1(\Omega)$-norm of $(\rk)^m$ is equal to $(m-1)\tU_m$ evaluated at $\rk$. 
Hence, we can bound the first term in the definition of the $BV(\Omega)$-norm uniformly by the classical estimates \eqref{eq:classicalbounds}. 
In order to estimate the variation of $(\rk)^m$ we estimate the term inside the supremum of the definition of $V((\rk)^m,\Omega)$. 
Thus let $\xi\in C^\infty_c(\Omega,\Rd)$ with $\left\|\xi\right\|_\infty \leq 1$, then we can use the discrete Euler-Lagrange equations \eqref{eq:DiscreteEL} to substitute
\begin{align}\label{eq:bvest}
\begin{split}
			\int_\Omega (\rk)^m \, \operatorname{div} (\xi) \dd x =& \int_{\Omega} \langle \xi(x), \nabla V \rangle \rk(x) + \langle \xi(x),\nabla W \ast \rk \rangle \rk(x)\dd x  + \frac{2}{\tau} \int_{\Omega^2} \langle \xi(x),x-y\rangle \ \mathrm{d} \bp_\tau^k(x,y) \\
			&  -\frac{1}{2\tau} \int_{\Omega^2} \langle \xi(x), x-z \rangle \dd \bq_\tau^k(x,z).
\end{split}
 \end{align}
By Assumption \ref{ass:VW} we have quadratic growth bounds for $\nabla V$ and $\nabla W$, so using the step size independent bounds on the second moment \eqref{eq:classicalbounds}, we can estimate the first terms in \eqref{eq:bvest} as follows:
\begin{align*}
	\int_{\Omega} \langle \xi(x), \nabla V \rangle \rk(x) + \langle \xi(x),\nabla W \ast \rk \rangle \rk(x) \dd x \leq 2 d_1 \left\|\xi\right\|_\infty  (1+\M(\rk)) \leq 2 d_1   (1+C).
\end{align*}
The second integral on the right-hand side of \eqref{eq:bvest} can be estimated using Jensen's inequality
\begin{align*}
	\left| \int_{\Omega^2} \langle \xi(x),x-y\rangle \ \mathrm{d} \bp_\tau^k(x,y) \right|
	\leq \left\|\xi\right\|_\infty  \left( \int_{\Omega^2} \left\|x-y\right\|^2 \dd \bp_\tau^k(x,y) \right)^{1/2} 
		\leq  \W(\rk,\rkm),
		\end{align*}
and similar for the third integral of the right-hand side of \eqref{eq:bvest}.
 Hence, we have the following upper bound for the variation of $(\rk)^m$:
 \begin{align*}
 V( (\rk)^m,\Omega) & \leq C \left(1 + \frac{\W(\rk,\rkm)}{\tau} + \frac{\W(\rk,\rkzm)}{\tau} \right).
	\end{align*}
In conclusion, the discrete solution $(\rk)_\kinN$ satisfies the desired bound \eqref{eq:bv}.
\end{proof}

\begin{thm}[Step size independent $L^2(0,T;BV(\Omega))$-estimate]\label{thm:bv}
Fix a time horizon $T>0$.
  There exists a constant $C$, depending only on $d_1$ to $d_4$ and $T$, 
  such that the corresponding interpolated solution $\bar\dens_\tau$ satisfies for each $\tau \in(0,\tau_*)$:
  \begin{align}\label{eq:bvestint}
  	\left\| (\bar\dens_\tau)^m \right\|_{L^2(0,T;BV(\Omega))} \leq C.
  \end{align}
\end{thm}

\begin{proof}
	We use the classical estimates on the kinetic energy \eqref{eq:classicalbounds} and the result from Proposition \ref{prop:bv} to estimate the $L^2(0,T;BV(\Omega))$-norm of $(\bar\dens_\tau)^m$. Let $N_T:=\max\{N \in\N\mid N\tau\leq T\}$, then we have 
	\begin{align*}
	\left\| (\bar\dens_\tau)^m \right\|_{L^2(0,T;BV(\Omega))}^2  \leq \sum_{k=1}^{N_T+1} \int_{(k-1)\tau}^{k\tau}  \left\|(\rk)^{m}\right\|_{\BV}^2 \dd t &\leq  C \sum_{k=1}^{N_T+1} \tau \left(1 + \frac{\W(\rk,\rkm)}{\tau} + \frac{\W(\rk,\rkzm)}{\tau} \right)^2.
	\end{align*}
	By the triangle inequality $\W(\rk,\rkzm)\leq \W(\rk,\rkm) + \W(\rkm,\rkzm)$	in combination with a Cauchy type inequality we obtain 
	\begin{align*}
	\left\| (\bar\dens_\tau)^m \right\|_{L^2(0,T;BV(\Omega))}^2 &\leq C \sum_{k=1}^{N_T+1} \left[ \tau + \frac{\W^2(\rk,\rkm)}{\tau} + \frac{\W^2(\rkm,\rkzm)}{\tau} \right] \leq C(T +\tau) + C \sum_{k=0}^{N_T+1} \frac{\W^2(\rk,\rkm)}{\tau}.	
	\end{align*}
	Finally, we can conclude, under the step size independent bounds on the kinetic energy \eqref{eq:classicalbounds}, the desired estimate \eqref{eq:bvestint} for some universal constant $C$, which only depends on $d_1$ to $d_4$ and $T$, but not on the step size $\tau\in(0,\tau_*)$.
\end{proof}

\section{Convergence}\label{sec:Convergence}
 
 In this section we prove our main theorem, the strong convergence of the approximation $\bar\dens_\tau$ to the solution of the non-linear Fokker-Planck equation.
 The convergence in the strong $L^m(0,T;L^m(\Omega))$-topology follows by the improved $L^2(0,T;BV(\Omega))$-estimates \eqref{eq:bvestint} and by a general version of the Aubin-Lions Theorem \cite[Theorem 2]{rossi2003tightness}, which is recalled in Appendix \ref{app:Aubin}.

\begin{thm}[Strong convergence in $L^m(0,T;L^m(\Omega))$]\label{thm:convs}
Under the same assumptions as in Theorem \ref{thm:convn} and given the vanishing sequence $(\tau_n)_\ninN$ of step sizes $\tau_n\in(0,\tau_*)$ and the limit curve $\dens_*$ therein,
then there exists a further (non-relabelled) subsequence $(\tau_n)_\ninN$   such that for all $T>0$:
\begin{enumerate}
		\item[(a)] In the case of an open and bounded set $\Omega \subset\Rd$ with Lipschitz-continuous boundary $\partial \Omega$, we have:		\begin{align*}
			\bar\dens_{\tau_n}(t) \rightarrow \dens_*(t) \qquad &\text{in } L^m(0,T;L^m(\Omega)) \text{ as } n \to\infty.
		\end{align*}
		\item[(b)] In the case of the entire space, i.e., $\Omega=\Rd$, we have for every open and bounded set $\tilde\Omega \Subset \Rd$:
		\begin{align*}
			\bar\dens_{\tau_n}(t) \rightarrow \dens_*(t) \qquad &\text{in } L^m(0,T;L^m(\tilde\Omega)) \text{ as } n \to\infty.
		\end{align*} 
\end{enumerate}
\end{thm}

\begin{proof}[Proof of Theorem \ref{thm:convs} for $\Omega$ is open and  bounded with Lipschitz-continuous boundary $\partial\Omega$]
	Fix $T>0$. 	
	In order to prove the strong convergence result we use the Aubin-Lions Theorem \ref{thm:aubin} with the underlying Banach space $\X=L^m(\Omega)$. 
	We consider the functional $\calA:L^m(\Omega)\to\R$, defined via
	\begin{align*}
		\calA(\dens):=\begin{cases}
			\left\|\dens^m\right\|_{BV(\Omega)}^2 & \text{if } \dens\in\denses \ \text{and} \ \dens^m \in \BV, \\
			+\infty & \text{else}.
		\end{cases}
	\end{align*} 
	Using the remark in the introductory section about functions of bounded variations it follows that the functional $\calA$ is measurable, lower semi-continuous with respect to the $L^m(\Omega)$-topology, and has compact sublevels. 
	Next, we choose as pseudo-distance $g=\W$ on $L^m(\Omega)$. 
	The $L^2$-Wasserstein distance is lower semi-continuous with respect to the $L^m(\Omega)$-topology and clearly compatible with $\calA$.
	
	Next, we verify the assumption \eqref{eq:aubin} on $(\bar\dens_{\tau_n})_\ninN$ of Theorem \ref{thm:aubin}. By the $L^2(0,T;BV(\Omega))$-estimates of Theorem \ref{thm:bv} it is clear, that the sequence $(\bar\dens_{\tau_n})_\ninN$ is tight with respect to $\calA$, since we have:
	\begin{align*}
		\sup_\ninN \int_0^T \left\|(\bar\dens_{\tau_n}(t))^m\right\|_{BV(\Omega)}^2 \dd t = \sup_\ninN \left\| (\bar\dens_{\tau_n})^m \right\|^2_{L^2(0,T;BV(\Omega))} \leq C < \infty.
	\end{align*}
	For the proof of the relaxed averaged weak integral equicontinuity condition of $(\bar\dens_{\tau_n})_\ninN$ with respect to $\W$, we use the auxiliary function $A_n$ and the estimate \eqref{eq:hcont} from the proof of weak convergence results to obtain:
	\begin{align*}
	\int_0^{T-t} \W(\bar\dens_{\tau_n}(s+t),\bar\dens_{\tau_n}(s)) \dd s 
	\leq  \int_0^{T-t} \int_{(s-\tau_n)^+}^{s+t} A_n(r) \dd r \dd s 
	  \leq (t+\tau_n)  \int_0^T A_n (r)\dd r.
	\end{align*}
	Indeed, using the weak convergence in $L^2$ of $A_n$ to some $A\in L^2_{loc}(0,\infty)$ it follows
	\begin{align*}
	 \liminf_{h\searrow 0} \limsup_{n\to \infty }	 \frac1{h} \int_0^h \int_0^{T-t} \W(\bar\dens_{\tau_n}(s+t),\bar\dens_{\tau_n}(s)) \dd s \dd t 
	\leq   \liminf_{h\searrow 0} \limsup_{n\to \infty } \frac1{h} \int_0^h (t+\tau_n)  \dd t  \int_0^T A_n(t)\dd t =0.
	\end{align*}
	Therefore, we can conclude that there exists a non-relabeled subsequence $(\tau_n)_\ninN$ such that $\bar\dens_{\tau_n}$ converges in $\mathcal{M}(0,T;\Lm)$ to some curve $\dens_+$.	
	Due to the uniform bounds in $L^\infty(0,T;L^m(\Omega))$, we obtain by a dominated convergence argument also convergence in $L^m(0,T;L^m(\Omega))$ as desired.
	Moreover, the limit curves $\dens_+$ and $\dens_*$ have to coincide, since $\bar\dens_{\tau_n}$ converges also in measure to $\dens_+$ and $\dens_*$, so both limits have to be equal.
\end{proof}

In the case of $\Omega=\Rd$ we have to alter the proof given above, since the embedding of $BV(\Rd)$ into $L^1(\Rd)$ is not compact anymore. So we restrict ourself to the open and bounded sets $\tilde\Omega = \mathbb{B}_R(0)$. 
This subset is clearly open and bounded with Lipschitz-continuous boundary $\partial\tilde\Omega$, so the embedding of $BV(\tilde\Omega)$ into $L^1(\tilde\Omega)$ is compact again. 

\begin{proof}[Proof of Theorem \ref{thm:convs} for $\Omega=\Rd$]
	Fix $T>0$. 	
	Without loss of generality we can assume $\tilde\Omega = \mathbb{B}_R(0)$, since every open and bounded subset $\hat\Omega \Subset \Rd$ is contained in a ball with radius $R$ and convergence in $L^m(0,T;L^m(\mathbb{B}_R(0)))$ implies convergence in $L^m(0,T;L^m(\hat\Omega))$. \\
	
	As before, we want to use the Aubin-Lions Theorem \ref{thm:aubin} for the Banach space $L^m(\tilde\Omega)$ equipped with the natural topology induced by the $L^m(\tilde\Omega)$-norm applied to $ (\left.\bar\dens_{\tau_n}\right|_{\tilde\Omega})_\ninN$, the restriction of the density $\bar\dens_{\tau_n}$ to the subspace $\tilde\Omega$.
	In this case we consider the functional $\tilde\calA:L^m(\tilde\Omega)\to\R$, defined via
		\begin{align*}
		\tilde\calA(\dens):=\begin{cases}
			\left\|\dens^m\right\|_{BV(\tilde\Omega)}^2 & \text{if } \dens\in\mathcal{M}_f(\tilde\Omega) \ \text{and} \ \dens^m \in BV(\tilde\Omega), \\
			+\infty & \text{else}.
		\end{cases}
	\end{align*} 
	Now, the functional $\tilde\calA$ is measurable, lower semi-continuous with respect to the $L^m(\tilde\Omega)$ topology, and has compact sublevels. Since $\tilde\calA(\left.\dens\right|_{\tilde\Omega}) \leq \calA(\dens)$, we obtain by the same calculations as above the tightness of $ (\left.\bar\dens_{\tau_n}\right|_{\tilde\Omega})_\ninN$ with respect to $\tilde\calA$.
	
	Since the measure $\left.\dens\right|_{\tilde\Omega}$ does not have unit mass anymore, we cannot consider the $L^2$-Wasserstein distance $\W$ as pseudo-distance anymore.
	However, we can use the following pseudo-distance $\tilde g$:
	\begin{align*}
		\tilde g(\dens,\nu):= \inf \left\{ \W(\tilde{\dens},\tilde{\nu}) \mid \tilde\dens \in \Sigma(\dens), \ \tilde\nu \in \Sigma(\nu)\right\}, \quad \Sigma(\dens):=\left\{ \tilde\dens \in \mathscr{P}(\Rd) \mid \left. \tilde\dens \right|_\Omega =\dens, \M(\tilde\dens) \leq C\right\},
	\end{align*}
	where $C$ is the constant from the classical estimates \eqref{eq:classicalbounds} for the specific $T$. 
	Since $\Sigma(\dens)$ and $\Sigma(\nu)$ are compact sets with respect to the narrow topology, the infimum is attained at some pair $\tilde \dens_*,\tilde \nu_*$.
The pseudo-distance $\tilde{g}$ is compatible with $\tilde\calA$, i.e., if $\dens^m,\nu^m\in BV(\tilde\Omega)$ and $\tilde g(\dens,\nu)=0$ then $\dens=\nu$ a.e. on $\tilde\Omega$.
	The lower semi-continuity of the pseudo-distance $\tilde{g}$ with respect to the $L^m(\tilde\Omega)$-topology can be proven as follows. 
	Choose to convergent sequences $\dens_n\rightarrow\dens$ and $\nu_n\to \nu$ in $L^m(\tilde\Omega)$ with $\sup_n \tilde g(\dens_n,\nu_n)<\infty$. 
	By the remark from above, there exists $\tilde\dens_n,\tilde\nu_n$ such that $\tilde g (\dens_n,\nu_n)=\W(\tilde\dens_n,\tilde\nu_n)$. 
	Since the second moments are by definition of $\Sigma(\dens)$ uniformly bounded, we can extract a non-relabeled convergent subsequence which converges narrowly to $\tilde\dens\in\Sigma(\dens),\tilde\nu\in\Sigma(\nu)$. 
	By the lower semi-continuity of $W$ with respect to narrow convergence, we get in the end
	\begin{align*}
		\tilde g(\dens,\nu) \leq \W(\tilde \dens,\tilde \nu) \leq \liminf_{\ntoinf} \W(\tilde \dens_n, \tilde \nu_n) = \liminf_\ntoinf \W(\dens_n,\nu_n).
	\end{align*}
	Therefore, the pseudo-distance $\tilde g$ is lower semi-continuous with respect to the $L^m(\tilde\Omega)$-topology.
	Thus, $\tilde g$ satisfies the assumptions of theorem \ref{thm:aubin}.
	Further, one has $\tilde g(\left.\dens\right|_{\tilde\Omega},\left.\nu\right|_{\tilde\Omega}) \leq \W(\dens,\nu)$.
	Thus we derive, using the same proof as above, the equicontinuity of $ (\left.\bar\dens_{\tau_n}\right|_{\tilde\Omega})_\ninN$ with respect to the pseudo-distance $\tilde g$.
	
	Hence, we can conclude that there exists a non-relabeled subsequence of $\left.\bar\dens_{\tau_n}\right|_{\tilde\Omega}$ which converges in $\mathcal{M}(0,T;L^m(\tilde\Omega))$ to some limit $\dens_+$.
	As before, we use the uniform bounds in $L^\infty(0,T;L^m(\tilde\Omega))$, to obtain the strong convergence in $L^m(0,T;L^m(\tilde\Omega))$ by a dominated convergence argument.
	Moreover, the limit curves $\dens_+$ and $\left.\dens_*\right|_{\tilde\Omega}$ have to coincide on $\tilde\Omega$, since $\left.\bar\dens_{\tau_n}\right|_{\tilde\Omega}$ converges also in measure on $\tilde\Omega$ to $\dens_+$ and $\left.\dens_*\right|_{\tilde\Omega}$, so both limits have to be equal on $\tilde\Omega$.
	Two diagonal arguments in $T\to\infty$ and $R\to \infty$ yield the desired convergence result.
	\end{proof}

To complete the proof of the main theorem \ref{thm:mainthm}, we have to validate that $\dens_*$ is indeed a solution to \eqref{eq:FP} in the sense of distributions.
\begin{thm}[Solution of the non-linear Fokker-Planck equation]
Under the same assumptions as in Theorem \ref{thm:convs}, consider the vanishing sequence $(\tau_n)_\ninN$ of step sizes $\tau \in(0,\tau_*)$ and the limit curve $\dens_*$ defined there.
  The limit curve $\dens_*$ is a solution to the non-linear Fokker-Planck equation with no-flux boundary condition \eqref{eq:FP} in the sense of distributions, i.e., we have for each test function $\psi\in C^\infty_c([0,\infty)\times\Omega)$:
	\begin{align*}
		 \int_0^\infty \int_\Omega - \Delta\psi \, \dens_*^m + \langle\nabla\psi,\nabla V \rangle \, \dens_*  + \langle\nabla\psi,\nabla W \ast \dens_*\rangle \, \dens_* \dd x \dd t 
		=& \int_0^\infty \int_\Omega \partial_t \psi \, \dens_* \dd x\dd t+ \int_\Omega  \psi(0) \, \dens^0 \dd x.
	\end{align*}
\end{thm}

\begin{proof} 
For simplicity we drop the index $n$ and write $\tau$ and $\tau \searrow  0$. Fix $\psi\in C_c^\infty([0,\infty)\times\Omega)$ and let be $T>0$ and $\tilde\Omega \subset\Omega$ be open and bounded such that $\mathrm{supp}\ \psi\subset [0,T]\times\tilde\Omega$.
Further, define the piecewise constant interpolation $\bar\psi$ of $\psi$ by 
\begin{align*}
	\bar\psi(0)=\psi(0),  \quad \bar\psi(t)=\psi(k\tau)\quad  \text{for } t\in((k-1)\tau,k\tau] \quad \text{and } \kinN.
\end{align*}

For each $\kinN$ insert the smooth function $x\mapsto\nabla \psi((k-1)\tau,x)$ in the discrete Euler-Lagrange equation \eqref{eq:DiscreteEL} for the vector field $\xi\in C^\infty_c(\Omega)$. Summing the resulting equations from $k=1$ to $N_T+1$ and multiplying with $\tau$ yields:
\begin{align*}
		0=& \int_0^T \int_{\tilde\Omega} - \Delta\bar\psi \, (\bar\dens_{\tau})^m + \langle \nabla\bar\psi , \nabla V \rangle \,  \bar\dens_{\tau} + \langle \nabla\bar\psi ,\nabla W \ast \bar\dens_{\tau} \rangle  \, \bar\dens_{\tau} \dd x \dd t \\
		& +\sum_{k=1}^{N_T} \left[ 2 \int_{\tilde\Omega^2} \langle \nabla \psi ((k-1)\tau,x),x-y \rangle   \ \mathrm{d} \bp_\tau^k(x,y) -\frac{1}{2} \int_{\tilde\Omega^2} \langle \nabla \psi ((k-1)\tau,x),x-z \rangle \ \mathrm{d} \bq_\tau^k(x,z) \right]\\
		&=: I_1 + I_2.
\end{align*}
Due to the strong convergence in $L^m(0,T;L^m(\tilde\Omega))$ of $\bar\dens_\tau$ to $\dens_*$ and due to the uniform convergence of $\bar\psi$ to $\psi$\begin{align*}
	\lim_{\tau \searrow  0} I_1=	  &   \int_0^T \int_{\tilde\Omega} - \Delta\psi \, \dens_*^m + \langle \nabla \psi, \nabla V \rangle \, \dens_* + \langle \nabla \psi,  \nabla W \ast \dens_* \rangle \, \dens_* \dd x \dd t.
\end{align*} 
To rewrite $I_2$, we use, as in \cite{jko}, the second order Taylor expansion for a time independent function $\psi$, to obtain
\begin{align*}
& \left| \int_{\tilde\Omega} \psi(y) \rkm(y)  \dd y  - \int_{\tilde\Omega}  \psi(x) \rk(x) \dd x   - \int_{{\tilde\Omega}^2} \langle \nabla \psi(x), y-x \rangle  \ \mathrm{d} \bp_\tau^k(x,y) \right| \\
= & \left| \int_{\tilde\Omega^2} \psi(y) -  \psi(x)  -  \langle \nabla \psi(x), y-x \rangle  \ \mathrm{d} \bp_\tau^k(x,y) \right| \\
\leq &  \frac12 \left\| \operatorname{Hess} \psi \right\|_\infty  \int_{\tilde\Omega^2} \left\|x-y\right\|^2 \mathrm{d} \bp_\tau^k(x,y) \\
= & \frac12 \left\| \operatorname{Hess} \psi \right\|_\infty  \W^2(\rk,\rkm).
	\end{align*}
Replacing the time independent function $\psi$ with $\psi((k-1)\tau)$ yields as approximation of $I_2$
	\begin{align*}
 I_2 = &   \sum_{k=1}^{N_T} \left[\int_{\tilde\Omega} (\frac32\rk -2\rkm + \frac12 \rho_\tau^{k-2})\psi((k-1)\tau) \dd x +  \mathcal{O} ( \W^2(\rk,\rkm)) + \mathcal{O}(\W^2(\rk,\rho_\tau^{k-2}) ) \right].
\end{align*}
We rearrange the sum of the first term in $I_2$ as follows
\begin{align*}
 & \sum_{k=0}^{N_T} \int_{\tilde\Omega} (\frac32\rk -2\rkm + \frac12 \rho_\tau^{k-2})\psi((k-1)\tau) \dd x \\
 = & \sum_{k=0}^{N_T} \int_{\tilde\Omega}  ( \frac 32 \psi((k-1)\tau) - 2 \psi(k\tau) + \frac12 \psi((k+1)\tau) ) \, \rk \dd x - \int_{\tilde\Omega}  (2 \psi(0) - \frac12 \psi(\tau))\, \dens_\tau^0 - \frac12 \psi(0) \, \dens_\tau^{\minus 1} \dd x.
 \end{align*}
Finally, use the fundamental theorem of calculus and the classical estimate \eqref{eq:classicalbounds} to bound the second term in $I_2$, to obtain
\begin{align*}
	I_2	= & - \int_0^T \int_{\tilde\Omega}(\frac32 \partial_t \psi(t) -\frac12 \partial_t \psi(t+\tau)) \, \bar\dens_{\tau}  \dd x \dd t - \int_{\tilde\Omega}  (2 \psi(0) - \frac12 \psi(\tau))\, \dens_\tau^0 - \frac12 \psi(0) \, \dens_\tau^{\minus 1} \dd x +  \mathcal{O} ( \tau ).
\end{align*}
Indeed, combining the narrow convergence of $\bar\dens_\tau$ with the uniform convergence of $\partial_t\psi(t+\tau)$ to $\partial_t\psi(t)$ and with the narrow convergence of the initial data $(\dens_\tau^0,\dens_\tau^{\minus 1})$ to $\dens^0$, the limit of $I_2$ is given by:
\begin{align*}
	\lim_{\tau \searrow 0} I_2
	= & - \int_0^T \int_{\tilde\Omega}  \partial_t \psi(t) \, \dens_*  \dd x \dd t - \int_{\tilde\Omega} \psi(0) \, \dens^0  \dd x.
\end{align*}
Finally, we can conclude that for arbitrary test functions $\psi\in C^\infty_c([0,\infty)\times\Omega)$ the limit curve $\dens_*$ satisfies :
	\begin{align*}
		 \int_0^\infty \int_\Omega - \Delta\psi \dens_*^m + \langle\nabla\psi,\nabla V \rangle \, \dens_*+ \langle   \nabla\psi,\nabla W \ast \dens_* \rangle \, \dens_* \dd x \dd t 
		=& \int_0^\infty \int_\Omega \partial_t \psi \, \dens_* \dd x\dd t+ \int_\Omega \psi(0)\, \dens^0  \dd x.
	\end{align*}
This yields that $\dens_*$ is a distributional solution to the non-linear Fokker-Planck equation \eqref{eq:FP}.
\end{proof}

\begin{appendix}

\section{Proof of the Classical Estimates}\label{app:Bounds}
In this part we fill the technical gap in the proof of the step size independent bounds of Theorem \ref{thm:bounds}. 
The first result is an auxiliary inequality which will be used to derive the step size independent bounds.
Despite the auxiliary character of this inequality, we want to emphasize that this property is of interest by itself, since we can give a precise estimate of the energy decay of the BDF2 scheme in every step.

 \begin{lemma}[Almost energy diminishing]\label{lem:EnergyDim}
For each time step size $\tau\in(0,\tau_*)$ the discrete solution $(\rk)_{k\in\N}$ satisfies
  \begin{align}
    \nrg(\rk) + \frac1{2\tau} \W^2(\rkm,\rk) \leq \nrg(\rkm) + \frac{1}{4\tau} \W^2(\rkzm,\rkm)
    \label{eq:EnergyDim}
  \end{align}
  at each step $k=1,2,\ldots$.
\end{lemma}

\begin{proof}
  Since $\rk$ is a minimizer of $\dens \mapsto \Psi(\tau,\rkzm,\rkm; \dens)$, it satisfies 
  $$ \Psi(\tau,\rkzm,\rkm;\rk) \leq \Psi(\tau,\rkzm,\rkm;\dens)$$ 
  for all $\dens\in\denses$.
  For the specific choice $\dens= \rk$, we obtain
  \begin{align}
    \label{eq:predim}
    \frac{1}{\tau} \W^2(\rkm,\rk) - \frac{1}{4\tau} \W^2(\rkzm,\rk) + \nrg(\rk)  
    \leq - \frac{1}{4\tau} \W^2(\rkzm,\rkm) + \nrg(\rkm).
  \end{align}
  By the triangle inequality and the binomial formula,
  \begin{align}
    \W^2(\rkzm,\rk) \le 2\W^2(\rkzm,\rkm) + 2\W^2(\rkm,\rk). \label{eq:binomf}
  \end{align}
  Substitute this in the left-hand side of \eqref{eq:predim}.
  This yields \eqref{eq:EnergyDim}
\end{proof}

\begin{thm}[Classical estimates]
Fix a time horizon $T>0$.
There exists a constant $C$, depending only on $d_1$ to $d_4$ and $T$, 
  such that the corresponding discrete solutions $(\dens_\tau^k)_{k\in\N}$ satisfy
  \begin{align*}
    \sum_{k=0}^{N} \frac{1}{2\tau} \W^2(\rkm,\rk) & \leq  C ,  &   |\tU_m(\dens_\tau^{N}) | & \leq C,  & \M(\rN) &\leq C ,
  \end{align*}
  for all $\tau\in(0,\tau_*)$ and for all $N\in\N$ with $N\tau\le T$. 
 \end{thm}

\begin{proof}[Proof of Theorem \ref{thm:bounds}]
Sum up inequalities \eqref{eq:EnergyDim} for $k=1$ to $K=N$ to obtain after cancellation:
  \begin{align}
    \label{eq:bound1}
    \nrg(\dens_\tau^N)+\frac1{4\tau}\sum_{k=1}^N\W^2(\rkm,\rk)
    \le \nrg(\dens_\tau^0)+\frac1{4\tau}\W^2(\dens_\tau^{\minus 1},\dens_\tau^{0}).
  \end{align}
  Next, we want to prove the auxiliary inequality
  \begin{align}\label{eq:bound2}
  	\M^2(\rk) -\M^2(\rkm) \leq 2 \W(\rkm,\rk)\M(\rk).
  \end{align}
  Without loss of generality we assume $\M(\rkm)\geq\M(\rk)$, otherwise the equality is always true. 
  We use the binomial formula to obtain
  \begin{align*}
  	\M^2(\rk) -\M^2(\rkm) = (\M(\rk) +\M(\rkm))(\M(\rk) -\M(\rkm)) \leq 2 \M(\rk) (\M(\rk) -\M(\rkm))
  \end{align*}
 Let $\delta_0$ be the discrete probability measure localized at $x=0$, then we have by the triangle inequality 
  \begin{align*}
  	\M(\dens) = \W(\dens, \delta_0) \leq  \W(\dens,\nu) + \W(\nu,\delta_0) = \W(\dens,\nu) + \M(\nu).
  \end{align*}	
This yields \eqref{eq:bound2}. 
A Cauchy type inequality with \eqref{eq:bound2} yields 
  \begin{align*}
  	\M^2(\dens_\tau^N)-\M^2(\dens_\tau^0) &= \sum_{k=1}^N \left[ \M^2(\rk)-\M^2(\rkm) \right] \leq  2 \sum_{k=1}^N \W(\rkm,\rk)\M(\rk)\\
  	&  \leq \frac{\tau_*}{4} \sum_{k=1}^N \frac{\W^2(\rkm,\rk)}{\tau} + \frac{4\tau}{\tau_*} \sum_{k=1}^N \M^2(\rk).
  \end{align*} 
  Substitute \eqref{eq:bound1} into this inequality:
  \begin{align*}
  	\M^2(\dens_\tau^N) \leq \M^2(\dens_\tau^0) + \tau_* \left( \nrg(\dens_\tau^0)+\frac1{4\tau}\W^2(\dens_\tau^{\minus 1},\dens_\tau^{0}) -\nrg(\dens_\tau^N) \right) + \frac{4\tau}{\tau_*} \sum_{k=1}^N \M^2(\rk).
  \end{align*}
  The first term of the right-hand side is estimated by
  \begin{align}\label{eq:eq1}
  	\M(\dens_\tau^0)\leq 2 \W(\dens_\tau^0,\dens^0) + \M(\dens^0) \leq 2 d_3 \sqrt\tau + 2\M(\dens^0).
  \end{align}
  The energy $\nrg$ evaluated at $\dens_\tau^0$ is estimated using Assumptions \ref{ass:VW} and \ref{ass:I} and estimate \eqref{eq:eq1}:
  \begin{align}
  	\nrg(\dens_\tau^0)= \tU_m(\dens_\tau^0) + \tV(\dens_\tau^0) + \tW(\dens_\tau^0) \leq d_4 + \frac32d_1(1+ \M(\dens_\tau^0)) \leq d_4 + \frac32d_1(1+ 2 d_3 \sqrt\tau + 2\M(\dens^0)). \label{eq:eq2}
  	  	\end{align}
  A lower bound of the energy $\nrg$ evaluated at $\dens_\tau^N$ is derived by the same way as in the prove of Lemma \ref{lem:lower}, i.e., there exist constants $d_2$ and $\gamma \in(\frac{d}{d+1},1)$ such that
    \begin{align*}
  	\nrg(\dens_\tau^N) \geq - d_2(1+\M(\dens_\tau^N))^\gamma- \frac32 d_1(1+\M(\dens_\tau^N)) \geq - ( d_2 + \frac32 d_1)( 2+\M^2(\dens_\tau^N))  .
  \end{align*}
  Hence, there is a universal constant $C$, not depending on the step size $\tau$, such that
  \begin{align*}
  	 	\M^2(\dens_\tau^N) \leq C + \tau_*( d_2 + \frac32 d_1) \M^2(\dens_\tau^N)  + \frac{4\tau}{\tau_*} \sum_{k=1}^N \M^2(\rk).
  \end{align*}
  We rearrange terms and use the definition of \eqref{eq:taustar} to arrive at the time-discrete Gronwall inequality
   	\begin{align*}
   		\M^2(\dens_\tau^N) \leq 2 C   + \frac{8\tau}{\tau_*} \sum_{k=1}^N \M^2(\rk).
   	\end{align*}
  By induction on $N$ we obtain 
  \begin{align*}
  	\M^2(\dens_\tau^N) \leq C \left(1+ \frac{8\tau}{\tau_*}\right)^N \leq \hat C \exp\left( \frac{8 N \tau}{\tau_*} \right) \leq \hat C \exp \left( \frac{8 T}{\tau_*} \right).
  \end{align*}
So the second moments $\M(\dens_\tau^N)$ of the discrete solution are uniformly bounded independent of the step size $\tau \in(0,\tau_*)$ and for all $N\in\N$ with $N\tau<T$. 

The remaining estimates can be derived from this. 
An upper bound for the energy $\nrg(\dens_\tau^N)$ follows from \eqref{eq:bound1} and \eqref{eq:eq2} combined with Assumption \ref{ass:VW}.
The lower bound on $\nrg(\dens_\tau^N)$ follows by the lower bounds on $\tU_m, \tV$, and $\tW$ in terms of the second moment.
Hence, the boundedness of $\nrg(\dens_\tau^N),\tV(\dens_\tau^N)$, and $\tW(\dens_\tau^N)$ yields the boundedness of $\tU_m(\dens_\tau^N)$.
The upper bound for the kinetic energy follows from the lower bound for the energy $\nrg(\dens_\tau^N)$, \eqref{eq:bound1}, and \eqref{eq:eq2} combined with Assumption \ref{ass:VW}.
\end{proof}

\section{Auxiliary theorems}\label{app:Aubin}
The following theorem is an extension of the Aubin-Lions Theorem for Banach-spaces proven in \cite{rossi2003tightness}. This theorem is the main tool to prove the convergence of the interpolated solution. 

\begin{thm}[Extension of the Aubin-Lions Theorem] \label{thm:aubin}
Let $\X$ be a separable Banach space, $\mathcal{A}:\X \rightarrow \Ru$ be measurable, lower semi-continuous and with compact sublevels in $\X$, and $g: \X\times \X \rightarrow \Ru$ be lower semi-continuous and such that $g(u,v)=0$ for $u,v \in \mathcal{D}( \mathcal{A})$ implies $u=v$. Let $\left(u_n\right)_\ninN$ be a sequence of measurable functions $u_n:(0,T)\to\X$ such that
	\begin{align}\label{eq:aubin}
		\sup_{\ninN} \int_0^T \calA(u_n(t)) \dd t< \infty, \qquad  \lim_{h \searrow 0} \limsup_{n\to\infty} \frac{1}{h}\int_0^h  \int_0^{T-t} g\left( u_n(s+t),u_n(s) \right) \dd s \dd t =0.
	\end{align}
Then, $\left(u_n\right)_\ninN$ possesses a subsequence converging in $\mathcal{M} \left(0,T;\X\right)$ to a limit curve $u_*:[0,T]\to \X$, i.e., 
\begin{align*}
	\lim_\ntoinf \mathcal{L}^1( \left\{ t\in(0,T) \mid \left\| u_n(t) - u_*(t) \right\|_\X \geq \sigma \right\}) = 0 \qquad \text{for all } \sigma >0.
\end{align*}
\end{thm}

\begin{remark}
	Note we replaced the usual weak integral equicontinuity condition
	\begin{align*}
		\lim_{h \searrow 0} \sup_{\ninN}  \int_0^{T-h} g\left( u_n(t+h),u_n(t) \right)\dd t =0,
	\end{align*}
	given in the original version of the theorem. In the proof it is sufficient to have the relaxed averaged weak integral equicontinuity, given in the theorem above. 
\end{remark}
	
\end{appendix}

\bibliographystyle{abbrv}
   \bibliography{thlit}

\def\cprime{$'$}
\begin{thebibliography}{10}

\bibitem{ambrosio2000functions}
L.~Ambrosio, N.~Fusco, and D.~Pallara.
\newblock {\em Functions of bounded variation and free discontinuity problems},
  volume 254.
\newblock Clarendon Press Oxford, 2000.

\bibitem{ags}
L.~Ambrosio, N.~Gigli, and G.~Savar{\'e}.
\newblock {\em Gradient flows in metric spaces and in the space of probability
  measures}.
\newblock Lectures in Mathematics ETH Z\"urich. Birkh\"auser Verlag, Basel,
  second edition, 2008.

\bibitem{benamou2016discretization}
J.-D. Benamou, G.~Carlier, Q.~M{\'e}rigot, and E.~Oudet.
\newblock Discretization of functionals involving the monge--amp{\`e}re
  operator.
\newblock {\em Numerische Mathematik}, 134(3):611--636, 2016.

\bibitem{BCCkellersegel}
A.~Blanchet, V.~Calvez, and J.~A. Carrillo.
\newblock Convergence of the mass-transport steepest descent scheme for the
  subcritical {P}atlak--{K}eller--{S}egel model.
\newblock {\em SIAM Journal on Numerical Analysis}, 46(2):691--721, 2008.

\bibitem{blanchet2012}
A.~Blanchet and P.~Lauren{\c{c}}ot.
\newblock The parabolic-parabolic {K}eller-{S}egel system with critical
  diffusion as a gradient flow in {$\mathbb{R}^d,\ d\ge3$}.
\newblock {\em Comm. Partial Differential Equations}, 38(4):658--686, 2013.

\bibitem{calvez2016blow}
V.~Calvez and T.~Gallou{\"e}t.
\newblock Blow-up phenomena for gradient flows of discrete homogeneous
  functionals.
\newblock {\em arXiv preprint arXiv:1603.05380}, 2016.

\bibitem{carrillo2016convergence}
J.~Carrillo, F.~Patacchini, P.~Sternberg, and G.~Wolansky.
\newblock Convergence of a particle method for diffusive gradient flows in one
  dimension.
\newblock {\em SIAM Journal on Mathematical Analysis}, 48(6):3708--3741, 2016.

\bibitem{carrillo2015numerical}
J.~Carrillo, H.~Ranetbauer, and M.~Wolfram.
\newblock Numerical simulations of nonlinear convectionaggregation equataions
  by evolving diffeomorphisms.
\newblock {\em preprint}, 2015.

\bibitem{carrillo2011}
J.~A. Carrillo, M.~DiFrancesco, A.~Figalli, T.~Laurent, and D.~Slep{\u{c}}ev.
\newblock Global-in-time weak measure solutions and finite-time aggregation for
  nonlocal interaction equations.
\newblock {\em Duke Math. J.}, 156(2):229--271, 02 2011.

\bibitem{carrillo2009numerical}
J.~A. Carrillo and J.~S. Moll.
\newblock Numerical simulation of diffusive and aggregation phenomena in
  nonlinear continuity equations by evolving diffeomorphisms.
\newblock {\em SIAM Journal on Scientific Computing}, 31(6):4305--4329, 2009.

\bibitem{di2018nonlinear}
M.~Di~Francesco, A.~Esposito, and S.~Fagioli.
\newblock Nonlinear degenerate cross-diffusion systems with nonlocal
  interaction.
\newblock {\em Nonlinear Analysis}, 169:94--117, 2018.

\bibitem{during2010gradient}
B.~D{\"u}ring, D.~Matthes, and J.~P. Mili{\v{s}}ic.
\newblock A gradient flow scheme for nonlinear fourth order equations.
\newblock {\em Discrete Contin. Dyn. Syst. Ser. B}, 14(3):935--959, 2010.

\bibitem{erbar2010heat}
M.~Erbar et~al.
\newblock The heat equation on manifolds as a gradient flow in the
  {W}asserstein space.
\newblock In {\em Annales de l'Institut Henri Poincar{\'e}, Probabilit{\'e}s et
  Statistiques}, volume~46, pages 1--23. Institut Henri Poincar{\'e}, 2010.

\bibitem{gallouet2016lagrangian}
T.~Gallou{\"e}t and Q.~M{\'e}rigot.
\newblock A {L}agrangian scheme for the incompressible euler equation using
  optimal transport.
\newblock {\em arXiv preprint arXiv:1605.00568}, 2016.

\bibitem{GOthinfilm}
L.~Giacomelli and F.~Otto.
\newblock Variatonal formulation for the lubrication approximation of the
  {H}ele-{S}haw flow.
\newblock {\em Calculus of Variations and Partial Differential Equations},
  13(3):377--403, 2001.

\bibitem{gianazza2009}
U.~Gianazza, G.~Savar{\'e}, and G.~Toscani.
\newblock The {W}asserstein gradient flow of the {F}isher information and the
  quantum drift-diffusion equation.
\newblock {\em Arch. Ration. Mech. Anal.}, 194(1):133--220, 2009.

\bibitem{giusti1984minimal}
E.~Giusti.
\newblock {\em Minimal surfaces and functions of bounded variation}, volume~80.
\newblock Birkhauser Verlag, 1984.

\bibitem{jko}
R.~Jordan, D.~Kinderlehrer, and F.~Otto.
\newblock The variational formulation of the {F}okker-{P}lanck equation.
\newblock {\em SIAM J. Math. Anal.}, 29(1):1--17, 1998.

\bibitem{junge2015fully}
O.~Junge, D.~Matthes, and H.~Osberger.
\newblock A fully discrete variational scheme for solving nonlinear
  fokker-planck equations in higher space dimensions.
\newblock {\em arXiv preprint arXiv:1509.07721}, 2015.

\bibitem{Kinderlehrer2016}
D.~Kinderlehrer, L.~Monsaingeon, and X.~Xu.
\newblock A {W}asserstein gradient flow approach to {P}oisson-{N}ernst-{P}lanck
  equations.
\newblock {\em {ESAIM}: Control, Optimisation and Calculus of Variations},
  2016.

\bibitem{laguzet:hal-01404619}
L.~Laguzet.
\newblock {High order variational numerical schemes with application to Nash
  -MFG vaccination games}.
\newblock working paper or preprint, 2016.

\bibitem{laurencot2011}
P.~Lauren{\c{c}}ot and B.-V. Matioc.
\newblock A gradient flow approach to a thin film approximation of the {M}uskat
  problem.
\newblock {\em Calc. Var. Partial Differential Equations}, 47(1-2):319--341,
  2013.

\bibitem{LEGENDRE2017345}
G.~Legendre and G.~Turinici.
\newblock Second-order in time schemes for gradient flows in {W}asserstein and
  geodesic metric spaces.
\newblock {\em Comptes Rendus Mathematique}, 355(3):345 -- 353, 2017.

\bibitem{lisini2012}
S.~Lisini, D.~Matthes, and G.~Savar{\'e}.
\newblock Cahn-{H}illiard and thin film equations with nonlinear mobility as
  gradient flows in weighted-{W}asserstein metrics.
\newblock {\em J. Differential Equations}, 253(2):814--850, 2012.

\bibitem{MatthesMcCannSavare}
D.~Matthes, R.~J. McCann, and G.~Savar{\'e}.
\newblock A family of nonlinear fourth order equations of gradient flow type.
\newblock {\em Comm. Partial Differential Equations}, 34(10-12):1352--1397,
  2009.

\bibitem{matthes2014convergence}
D.~Matthes and H.~Osberger.
\newblock Convergence of a variational {L}agrangian scheme for a nonlinear
  drift diffusion equation.
\newblock {\em ESAIM: Mathematical Modelling and Numerical Analysis},
  48(3):697--726, 2014.

\bibitem{matthes2017variational}
D.~Matthes and S.~Plazotta.
\newblock A variational formulation of the {BDF2} method for metric gradient
  flows.
\newblock {\em arXiv preprint arXiv:1711.02935}, 2017.

\bibitem{matthes2017convergent}
D.~Matthes and B.~S{\"o}llner.
\newblock Convergent {L}agrangian discretization for drift-diffusion with
  nonlocal aggregation.
\newblock In {\em Innovative Algorithms and Analysis}, pages 313--351.
  Springer, 2017.

\bibitem{matthes2017existence}
D.~Matthes and J.~Zinsl.
\newblock Existence of solutions for a class of fourth order cross-diffusion
  systems of gradient flow type.
\newblock {\em Nonlinear Analysis}, 159:316--338, 2017.

\bibitem{otto2001}
F.~Otto.
\newblock The geometry of dissipative evolution equations: the porous medium
  equation.
\newblock {\em Comm. Partial Differential Equations}, 26(1-2):101--174, 2001.

\bibitem{peyre2015entropic}
G.~Peyr{\'e}.
\newblock Entropic approximation of {W}asserstein gradient flows.
\newblock {\em SIAM Journal on Imaging Sciences}, 8(4):2323--2351, 2015.

\bibitem{rossi2003tightness}
R.~Rossi and G.~Savar{\'e}.
\newblock Tightness, integral equicontinuity and compactness for evolution
  problems in {B}anach spaces.
\newblock {\em Annali della Scuola Normale Superiore di Pisa-Classe di
  Scienze-Serie V}, 2(2):395, 2003.

\bibitem{santambrogio2015optimal}
F.~Santambrogio.
\newblock {\em Optimal transport for applied mathematicians}.
\newblock Springer, 2015.

\bibitem{sturm2005convex}
K.-T. Sturm.
\newblock Convex functionals of probability measures and nonlinear diffusions
  on manifolds.
\newblock {\em Journal de math{\'e}matiques pures et appliqu{\'e}es},
  84(2):149--168, 2005.

\bibitem{topaz2006nonlocal}
C.~M. Topaz, A.~L. Bertozzi, and M.~A. Lewis.
\newblock A nonlocal continuum model for biological aggregation.
\newblock {\em Bulletin of mathematical biology}, 68(7):1601--1623, 2006.

\bibitem{villani}
C.~Villani.
\newblock {\em Topics in optimal transportation}, volume~58 of {\em Graduate
  Studies in Mathematics}.
\newblock American Mathematical Society, Providence, 2003.

\bibitem{villani2008optimal}
C.~Villani.
\newblock {\em Optimal transport: Old and New}, volume 338.
\newblock Springer Science \& Business Media, 2008.

\bibitem{westdickenberg2010variational}
M.~Westdickenberg and J.~Wilkening.
\newblock Variational particle schemes for the porous medium equation and for
  the system of isentropic euler equations.
\newblock {\em ESAIM: Mathematical Modelling and Numerical Analysis},
  44(1):133--166, 2010.

\bibitem{zinsl2015exponential}
J.~Zinsl and D.~Matthes.
\newblock Exponential convergence to equilibrium in a coupled gradient flow
  system modeling chemotaxis.
\newblock {\em Analysis \& PDE}, 8(2):425--466, 2015.

\bibitem{zinsl2015transport}
J.~Zinsl and D.~Matthes.
\newblock Transport distances and geodesic convexity for systems of degenerate
  diffusion equations.
\newblock {\em Calculus of Variations and Partial Differential Equations},
  54(4):3397--3438, 2015.

\end{thebibliography}

\end{document}